\documentclass[12pt,oneside,english]{amsart}
\usepackage[latin1]{inputenc}
\usepackage{amssymb}

\theoremstyle{plain}
\newtheorem{theorem}{Theorem}

\newtheorem{lemma}[theorem]{Lemma}

\theoremstyle{definition}

\newtheorem{example}[theorem]{Example}

\theoremstyle{remark}

\usepackage{babel}

\makeatother
\begin{document}
\baselineskip=17pt

\title[Identities generated by compositions of integers]
{On identities generated by compositions of positive integers}

\author{Vladimir Shevelev}
\address{Department of Mathematics \\Ben-Gurion University of the
 Negev\\Beer-Sheva 84105, Israel. e-mail: shevelev@bgu.ac.il}
 \newpage
\begin{abstract}
We prove astonishing identities generated by compositions of positive integers.
In passing, we obtain two new identities for Stirling numbers of the first kind.
In the two last sections we clarify an algebraic sense of these identities and obtain
several other structural close identities.
\end{abstract}
\maketitle
\section{Introduction}
Recall (cf.\cite {2}) that a composition of a positive integer $n$ is a way of
writing $n$ as a sum of a sequence of positive integers. These integers are called
parts of a composition. Thus to a composition of $n$ with $r$ parts corresponds $r$-
fold vector $(k_1,..., k_r)$ of positive integer components with the condition
 $k_1+k_2+...+k_r=n.$ From the definition it follows that, in contrast
 to partitions of $n,$ the order of parts matters. Note that the set
 of all solutions of the Diophantine equation  $k_1+k_2+...+k_r=n,\; k_i\geq1$
 is the set of all compositions with $r$ parts.
We start with two examples.
\begin{example}\label{e1}
Let $k=3.$ We have the following compositions of $3:$ $1+1+1=1+2=2+1=3.$
\end{example}
Let us map a composition $ k=k_1+k_2+...+k_r$ to the
 following product of binomial coefficients:
  $\binom{n}{k_1}\binom{n}{k_2}\cdot...\cdot\binom{n}{k_r}$ and all compositions
  of $k$ we map to the sum of such products, where the summand are taken
   with the sign $(-1)^{k-r}.$ After summing the products with the same sets of
   factors, we obtain a liner combinations of such products.
  In our case $k=3,$ we have the following linear combination of products of binomial coefficients:
  \begin{equation}\label{1}
  c_3(n)=\binom{n}{1}^3-2\binom{n}{1}\binom{n}{2}+\binom{n}{3}.
 \end{equation}
It is easy to verify that
\begin{equation}\label{2}
  c_3(n)=\binom{n+2}{3}.
 \end{equation}
 \begin{example}\label{e2}
 We have the following compositions of $k=4:$ $1+1+1+1=2+1+1=1+2+1=1+1+2=1+3=3+1=
 2+2=4.$
 \end{example}
  Thus we have the following linear combination of products of binomial coefficients:
\begin{equation}\label{3}
  c_4(n)=\binom{n}{1}^4-3\binom{n}{1}^2\binom{n}{2}+2\binom{n}{1}\binom{n}{3}+\binom{n}{2}^2
 -\binom{n}{4}
 \end{equation}
 and it is easy to verify that
\begin{equation}\label{4}
  c_4(n)=\binom{n+3}{4}.
  \end{equation}
 In general, we obtain the following.
\begin{theorem}\label{t3}
\begin{equation}\label{5}
\sum_{r=1}^k(-1)^{k-r}\sum_{k_1+k_2+...+k_r=k,\;k_i\geq1}
\prod_{i=1}^{r}\binom{n}{k_i}=\binom{n+k-1}{k}.
\end{equation}
\end{theorem}
In cases $k=3$ and $k=4,$  formula (\ref{5}), evidently, leads to Examples
\ref{e1}-\ref{e2}.
It is interesting to note (using a simple induction) that the $k$-th polynomial in $n$ of the sequence
$\{\binom{n+k-1}{k}\}$ is
the partial sum of values of the $(k-1)$-th one:
\begin{equation}\label{6}
\sum_{j=1}^n\binom{j+k-2}{k-1}=\binom{n+k-1}{k}.
\end{equation}

\section{An equivalent form of identity (5)}
We calculate the interior sum in (\ref{5}) in a combinatorial way. First, let us
 consider also zero parts in the compositions of $k.$ In this case we have the
 sum
 \begin{equation}\label{7}
\Sigma_1=\sum_{k_1+k_2+...+k_r=k,\;k_i\geq0}\prod_{i=1}^{r}\binom{n}{k_i}.
 \end{equation}
 To calculate this sum, suppose that we have $rn$ white points and mark $k$ from
 them. This we can do in $\binom{rn}{k}$ ways. On the other hand, we can mark $k_1$
 from $n$ points (since the white points are indistinguishable, we can choose
 any $n$ points), $k_2$ from another $n$ points, etc. Thus we immediately obtain the
 equality
  \begin{equation}\label{8}
 \Sigma_1=\binom{rn}{k}.
 \end{equation}
 To calculate  the required interior sum in (\ref{5})
 \begin{equation}\label{9}
\Sigma_2=\sum_{k_1+k_2+...+k_r=k,\;k_i\geq1}\prod_{i=1}^{r}\binom{n}{k_i},
 \end{equation}
 we should remove zero parts in $\Sigma_1$ (\ref{7}), using "include-exclude"
  formula. Hence, we find
  \newpage
  $$ \Sigma_2=\binom{rn}{k}-\binom{r}{1}\binom{(r-1)n}{k}+$$
\begin{equation}\label{10}
 \binom{r}{2}\binom{(r-2)n}{k}-...+(-1)^{r-1}\binom{r}{r-1}\binom{n}{k}.
 \end{equation}
Now, by (\ref{9})-(\ref{10}),  we see that (\ref{5}) is equivalent to the identity

$$\sum_{r=1}^k(-1)^{k-r}\sum_{j=0}^{r-1}(-1)^j\binom{r}{j}\binom{n(r-j)}{k}=
 \binom{n+k-1}{k},$$
 or, putting $i=r-j,$ to the identity
 \begin{equation}\label{11}
 \sum_{r=1}^k\sum_{i=1}^{r}(-1)^i\binom{r}{i}\binom{ni}{k}=
 (-1)^k\binom{n+k-1}{k}.
 \end{equation}
 Changing here the order of summing, we have
$$\sum_{i=1}^k\sum_{r=i}^{k}(-1)^i\binom{r}{i}\binom{ni}{k}=$$
 \begin{equation}\label{12}
\sum_{i=1}^k(-1)^i\binom{ni}{k}\sum_{r=i}^{k}\binom{r}{i}=(-1)^k\binom{n+k-1}{k}.
 \end{equation}
 As is well known,
 $$\sum_{r=i}^{k}\binom{r}{i}=\binom{k+1}{i+1}.$$
 Therefore, (\ref{5}) is equivalent to the identity:
  \begin{equation}\label{13}
\sum_{i=1}^k(-1)^i\binom{ni}{k}\binom{k+1}{i+1}=(-1)^k\binom{n+k-1}{k}.
 \end{equation}
\section{(13) as a polynomial identity in $n$}
Unfortunately, we are not able to give a direct inductive proof of (\ref{13}).
Note that (\ref{13}) means the equality between two polynomials in $n$ of degree
$k.$ Therefore, for a justification of (\ref{13}), it is natural to use Stirling
numbers of the first kind with the generating polynomial for them (\cite{1}):
\begin{equation}\label{14}
x(x-1)\cdot...\cdot(x-n+1)=\sum_{j=1}^n s(n,j)x^j,\;n\geq1.
\end{equation}

Writing (\ref{13}) in the form
$$\sum_{i=1}^k(-1)^i in(in-1)\cdot...\cdot(in-k+1)\binom{k+1}{i+1}=$$
\newpage
\begin{equation}\label{15}
(-1)^k(n+k-1)(n+k-2)\cdot...\cdot n,
\end{equation}
by (\ref{14}), we have
$$\sum_{i=1}^k(-1)^i\binom{k+1}{i+1}\sum_{t=1}^k s(k,t)(in)^t=$$
\begin{equation}\label{16}
(-1)^k\sum_{t=1}^ks(k,t)(n+k-1)^t.
\end{equation}

In the left hand side of (\ref{16}), the coefficient of $n^t$ equals
$$s(k,t)\sum_{i=0}^{k}(-1)^i\binom{k+1}{i+1}i^t=$$
$$-s(k,t)\sum_{j=1}^{k+1}(-1)^j
\binom{k+1}{j}(j-1)^t= $$
$$-s(k,t)\sum_{j=0}^{k+1}(-1)^j\binom{k+1}{j}(j-1)^t+s(k,t)(-1)^t.$$
Since $t\leq k,$ then the $(k+1)$-th difference
$$\Delta^{k+1}[(j-1)^t]=\sum_{j=0}^{k+1}(-1)^j\binom{k+1}{j}(j-1)^t=0$$
and we conclude that for $t\geq1$
\begin{equation}\label{17}
Coef_{n^t}(\sum_{i=1}^k(-1)^i in(in-1)\cdot...\cdot(in-k+1)\binom{k+1}{i+1})=
(-1)^ts(k,t).
\end{equation}
In the right hand side of (\ref{16}), the coefficient of $n^t$ equals
$$(-1)^k\sum_{j=0}^ks(k,j)Coef_{n^t}(n+k-1)^j= $$
$$(-1)^k\sum_{j=t}^ks(k,j)\binom{j}{t}(k-1)^{j-t}.$$
Thus, comparing with (\ref{17}), we conclude that identity (\ref{13})
is equivalent to the identity
\begin{equation}\label{18}
\sum_{j=t}^{k}\binom{j}{t}s(k,j)(k-1)^{j-t}=(-1)^{k+t}s(k,t).
\end{equation}
Further we need two lemmas.
\newpage
\section{Lemmas}

\begin{lemma}\label{L4} For $1\leq t\leq k,$ we have
\begin{equation}\label{19}
\sum_{j=t+1}^{k}\binom{j}{t}s(k,j)=ks(k-1,t).
\end{equation}
\end{lemma}
 \slshape Proof.\upshape \;We prove the lemma in the form:
\begin{equation}\label{20}
\sum_{i=1}^{k-t}\binom{t+i}{t}s(k,t+i)=ks(k-1,t),\; 1\leq t\leq k.
\end{equation}
We use induction over $k.$ Note that (\ref{20}) is
 valid for $k=1$ and $t\geq1.$ Suppose that
 \begin{equation}\label{21}
\sum_{i=1}^{k-1-t}\binom{t+i}{t}s(k-1,t+i)=(k-1)s(k-2,t),\; t\geq1,
\end{equation}
or, the same, changing the summing index $i:=i-1,$
\begin{equation}\label{22}
\sum_{i=2}^{k-t}\binom{t+i-1}{t}s(k-1,t+i-1)=(k-1)s(k-2,t),\; t\geq1,
\end{equation}
or
$$\sum_{i=1}^{k-t}\binom{t+i-1}{t}s(k-1,t+i-1)=$$
\begin{equation}\label{23}
(k-1)s(k-2,t)+s(k-1,t),\; t\geq1.
\end{equation}
For $t\geq2,$ put in (\ref{21}) $t:=t-1.$ Then, for $t\geq1,$ we have
 \begin{equation}\label{24}
\sum_{i=1}^{k-t}\binom{t+i-1}{t-1}s(k-1,t+i-1)=(k-1)s(k-2,t-1).
\end{equation}
This we sum with (\ref{23}). We find
$$\sum_{i=1}^{k-t}\binom{t+i}{t}s(k-1,t+i-1)=$$

\begin{equation}\label{25}
(k-1)s(k-2,t-1)+(k-1)s(k-2,t)+s(k-1,t)
,\; t\geq1.
\end{equation}
Recall that (\cite{1})
\begin{equation}\label{26}
s(n,t)=s(n-1,t-1)-(n-1)s(n-1,t).
\end{equation}
For $k\neq1,$ put here $n=k-1$ and multiply by $k-1.$ We have
$$(k-1)s(k-1,t)=$$ $$(k-1)s(k-2,t-1)-(k-1)(k-2)s(k-2,t)= $$
$$(k-1)s(k-2,t-1)-((k-1)^2-(k-1))s(k-2,t),$$
whence
\newpage
$$(k-1)^2s(k-2,t)=(k-1)s(k-2,t-1)-$$
\begin{equation}\label{27}
(k-1)s(k-1,t)+(k-1)s(k-2,t).
\end{equation}
Taking into account the inductive supposition (\ref{21}), from (\ref{27}) we find
$$(k-1)\sum_{i=1}^{k-1-t}\binom{t+i}{t}s(k-1,t+i)=$$
\begin{equation}\label{28}
(k-1)s(k-2,t-1)-(k-1)s(k-1,t)+(k-1)s(k-2,t).
\end{equation}
Note that, since $s(k-1,k)=0,$ then in (\ref{28}) we can consider the
summing up to $i=k-t.$ Subtracting (\ref{28}) from (\ref{25}), we have
$$\sum_{i=1}^{k-t}\binom{t+i}{t}(s(k-1,t+i-1)-(k-1)s(k-1, t+i))=ks(k-1,t).$$
Since
$$s(k-1,t+i-1)-(k-1)s(k-1, t+i)=s(k,t+i),$$

then we find
 $$\sum_{i=1}^{k-t}\binom{t+i}{t}s(k,t+i)=ks(k-1,t)$$
which, comparing with (\ref{21}), means the step of induction. \;\;\;\;$\square$
 \begin{lemma}\label{L5} We have
$$\sum_{i=1}^k(-1)^i(\binom{(n-1)i}{k}-\binom{ni}{k})\binom{k+1}{i+1}=$$
\begin{equation}\label{29}
\sum_{i=1}^{k-1}(-1)^i\binom{ni}{k-1}\binom{k}{i+1}.
\end{equation}
\end{lemma}
\slshape Proof.\upshape \; We prove (\ref{29}) in the form
 $$\sum_{i=1}^k(-1)^i\binom{(n-1)i}{k}\binom{k+1}{i+1}=$$
\begin{equation}\label{30}
\sum_{i=1}^k(-1)^i\binom{ni}{k}\binom{k+1}{i+1}+
\sum_{i=1}^{k-1}(-1)^i\binom{ni}{k-1}\binom{k}{i+1}.
\end{equation}
According to (\ref{17}) (which not depends on the validity of (\ref{13})),
the coefficient of $n^t$ of right hand side of (\ref{30})
 equals $\frac{(-1)^t}{k!}s(k,t)+\frac{(-1)^t}{(k-1)!}s(k-1,t).$ Thus, by (\ref{30}),
  we should prove that
 $$ Coef_{n^t}(\sum_{i=1}^k(-1)^i\binom{(n-1)i}{k}\binom{k+1}{i+1})=
 \frac{(-1)^t}{k!}(s(k,t)+ks(k-1,t),$$
 or
 \newpage
 $$ Coef_{n^t}(\sum_{i=1}^k(-1)^i\binom{k+1}{i+1})\sum_{r=0}^ks(k,r)((n-1)i)^r=$$ $$
 \sum_{i=1}^k(-1)^i\binom{k+1}{i+1})\sum_{r=t}^ks(k,r)i^r(-1)^{r-t}\binom{r}{t}=$$
 $$(-1)^t(s(k,t)+ks(k-1,t),$$
or, changing the order of summing, equivalently we should prove that
\begin{equation}\label{31}
\sum_{r=t}^k(-1)^r\binom{r}{t}s(k,r)\sum_{i=0}^k(-1)^i\binom{k+1}{i+1}i^r=
s(k,t)+ks(k-1,t)
\end{equation}
(we can sum over $i\geq0,$ since $r\geq t\geq1$).
Note that the interior sum of (\ref{31}) is
$$\sum_{i=0}^k(-1)^i\binom{k+1}{i+1}i^r=\sum_{j=1}^{k+1}(-1)^{j-1}
\binom{k+1}{j}(j-1)^r=$$

$$\sum_{j=0}^{k+1}(-1)^{j-1}\binom{k+1}{j}(j-1)^r+(-1)^r.$$
However, since $r\leq k,$ then
$$\Delta^{k+1}[(j-1)^r]=\sum_{j=0}^{k+1}(-1)^{j}\binom{k+1}{j}(j-1)^r=0$$
and thus
$$\sum_{i=0}^k(-1)^i\binom{k+1}{i+1}i^r=(-1)^r. $$
Now the left hand side of (\ref{31}) is $\sum_{r=t}^k\binom{r}{t}s(kr)$ and, by
Lemma \ref{L4}, is
$s(k,t)+ks(k-1,t).\;\;\;\;\;\square$

\section{Completion of proof of Theorem 3}
In Section 2 we proved that (\ref{5}) is equivalent to (\ref{13}).
Therefore, our aim is to prove (\ref{13}). We use induction over $k.$
 Note that (\ref{13}), evidently, satisfies in case $k=1$ and every $n.$
 Suppose that (\ref{13}) holds for $k:=k-1$
and every $n,$ i.e., \begin{equation}\label{32}
\sum_{i=1}^{k-1}(-1)^i\binom{ni}{k-1}\binom{k}{i+1}=(-1)^{k-1}\binom{n+k-2}{k-1}.
\end{equation}
By Lemma \ref{L5}, the inductive supposition (\ref{32}) is equivalent to the
identity
$$\sum_{i=1}^k(-1)^i(\binom{(n-1)i}{k}-\binom{ni}{k})\binom{k+1}{i+1}=$$
\begin{equation}\label{33}
(-1)^{k-1}\binom{n+k-2}{k-1}.
\end{equation}
\newpage
Putting $n:=j,$ and summing (\ref{33}) over $j$ from $j=1$ up to $j=n,$ according to
(\ref{6}), we find

$$\sum_{i=1}^{k}(-1)^i\binom{ni}{k}\binom{k+1}{i+1}=(-1)^{k}\binom{n+k-1}{k}$$
which is realized the step of induction.\;\;\;\;\; $\square$ \newline
Simultaneously, in view of the proved in Section 3 equivalence of
 (\ref{13}) and (\ref{18}), we proved the identity (\ref{18}).

\section{Remarks on the newness of identities (13), (18) and (19)}
  Formally, the identities (\ref{13}), (\ref{18}) and (\ref{19}) (and, consequently,
(\ref{5})) appear to be new, since they are absent in so fundamental sources
as \cite{1},\cite{4},\cite{7}. However, there is a deeper reason.
  The newness of (\ref{13}) (and together with it (\ref{18}) and (\ref{19})) is
  explained by the fact that there are no known
identities involving $\binom{in}{k}$ with the summing index $i.$ Indeed,
the only known generator of similar sums is Rothe-Hagen coefficient $A_k(x,n)$
 \cite{4}-\cite{5}. It is defined alternatively by the following formulas:
\begin{equation}\label{34}
A_k(x,n)=\frac{x}{x+kn}\binom{x+kn}{k},
\end{equation}
\begin{equation}\label{35}
A_k(x,n)=\sum_{i=0}^{k-1}(-1)^{i+k+1}\binom{k}{i}\binom{x+in}{k}\frac{x}{x+in}
,\; k\geq1.
\end{equation}
The comparison of these formulas leads to the identity of the form
$$\sum_{i=1}^{k-1}(-1)^{i+k+1}\binom{k}{i}\binom{x+in}{k}\frac{x}{x+in}=$$
\begin{equation}\label{36}
\frac{x}{x+kn}\binom{x+kn}{k}+(-1)^k\binom{x}{k}.
\end{equation}
Unfortunately, the attempt to eliminate from $x$ in $\binom{x+in}{k},$ putting $x=0,$
lead to the trivial identity $0=0.$ Consider another attempt. For $k>x\geq1,$ we have
$$\sum_{i=1}^{k-1}(-1)^{i+k+1}\binom{k}{i}\binom{x+in}{k}\frac{1}{x+in}=
\frac{1}{x+kn}\binom{x+kn}{k},$$
or
\begin{equation}\label{37}
\sum_{i=1}^{k}(-1)^{i-1}\binom{k}{i}\binom{x+in}{k}\frac{1}{x+in}=0,\; x\geq1.
\end{equation}
In the "\slshape singular \upshape " case $x=0,$ we obtain the required factor
of the form $\binom{ni}{k}$ and found (quite independently on (\ref{37})) a nice
identity
\newpage
\begin{equation}\label{38}
\sum_{i=1}^{k}\frac{(-1)^{i-1}}{i}\binom{in}{k}\binom{k}{i}=\frac{(-1)^{k-1}n}{k}
\end{equation}
which, most likely, is also new, but different from (\ref{13}). Indeed, denote
the left hand side of (\ref{38}) by $a_n(k).$ Using (\ref{14}),we have
$$a_n(k)=\frac{1}{n!}\sum_{i=1}^{n}\frac{(-1)^{i-1}}{i}\binom{n}{i}(ik)(ik-1)
\cdot...\cdot(ik-n+1)=$$
$$\frac{1}{n!}\sum_{i=1}^{n}\frac{(-1)^{i-1}}{i}\binom{n}{i}\sum_{t=0}^ns(n,t)
(ik)^t.$$
Thus, since $s(n,0)=0,$ then
\begin{equation}\label{39}
Coef_{k^t}(a_n(k))=\begin{cases} 0,\;\; if \; \;t=0,\\
\frac{s(n,\;t)}{n!}\sum_{i=1}^{n}(-1)^{i-1}\binom{n}{i}i^{t-1},\;\; if \;\;
t\geq1. \end{cases}
\end{equation}
Further, since
$$s(n,1)=(-1)^{n-1}(n-1)!,\;\sum_{i=1}^n(-1)^{i-1}\binom{n}{i}=1,$$
then
\begin{equation}\label{40}
Coef_{k}(a_n(k))=
\frac{(-1)^{n-1}}{n}.
\end{equation}
It is left to show that, for $t\geq2,$ we have
\begin{equation}\label{41}
s(n,\;t)\sum_{i=1}^{n}(-1)^{i-1}\binom{n}{i}i^{t-1}=0.
\end{equation}
Indeed, if $2\leq t\leq n,$ then we have
$$\sum_{i=1}^{n}(-1)^{i-1}\binom{n}{i}i^{t-1}=
(-1)^{n-1}\sum_{i=0}^{n}(-1)^{i}\binom{n}{i}(k-i)^{t-1}.$$
The latter is the $n$-th difference $\Delta^n[k^{t-1}]$ which, for $t\leq n,$
 equals 0. If $t>n,$ then $s(n,\;t)=0,$ and (\ref{41}) follows.\;\;\;\;\;\;\;\;\; $
 \square$
 \section{Gessel's short proof of (\ref{13})}
 Gessel \cite{3} proposed a short proof of the identity (\ref{13}).\newline
 \indent Let $P(x)$ be a polynomial of degree $k.$ Then, for the $(k+1)$-th difference of $P(x),$
we have
 $$\Delta^{k+1}[P(x)]=\sum_{j=0}^{k+1}(-1)^{k+1-j}\binom{k+1}{j}P(x+j)=0.$$
 In particular, for $x=0,$
 \newpage
 $$\sum_{j=0}^{k+1}(-1)^{j}\binom{k+1}{j}P(j)=0.$$
 Put here
 $$P(j)=P_{n,\;k}(j)=\binom{n(j-1)}{k}$$
 which is a polynomial in $j$ of degree $k.$ We have
  $$\sum_{j=0}^{k+1}(-1)^{j}\binom{k+1}{j}\binom{n(j-1)}{k}=0.$$
 Putting here $j-1=i,$ we find
  $$\sum_{i=-1}^k(-1)^{i}\binom{k+1}{i+1}\binom{ni}{k}=0, $$
  or, the same, for $k\geq1,$ we have
  $$\sum_{i=1}^k(-1)^{i}\binom{k+1}{i+1}\binom{ni}{k}=\binom{-n}{k}=$$
  \newline
  $$\frac{(-n)(-n-1)\cdot...\cdot(-n-(k-1))}{k!}=$$

   $$(-1)^k\frac{(n+k-1)(n+k-2)\cdot...\cdot n}{k!}=(-1)^k\binom{n+k-1}{k}.
  \;\;\;\;\square  $$
  \newline
  It is interesting to note that, if the author was successful to find such an elegant
  and simple proof, then, most likely, the identities (\ref{18}), (\ref{19})
  and (\ref{38}) were not discovered.
  \section{Dual case of identity (\ref{5})}
  Note that, together with Example\ref{e1}, we have the following identity
  $$\binom{n}{1}^3-2\binom{n}{1}\binom{n+1}{2}+\binom{n+2}{3}=\binom{n}{3}.$$
  In general, together with (\ref{5}), we prove the following dual identity.
  \begin{theorem}\label{t6}
\begin{equation}\label{42}
\sum_{r=1}^k(-1)^{k-r}\sum_{k_1+k_2+...+k_r=k,\;k_i\geq1}
\prod_{i=1}^{r}\binom{n+k_i-1}{k_i}=\binom{n}{k}.
\end{equation}
\end{theorem}

\begin{proof}
Again we calculate the interior sum in (\ref{42}) in a combinatorial way with firstly
 consideration also zero parts in the compositions of $k.$ In this case we have the
 sum
 \newpage
 \begin{equation}\label{43}
\Sigma_3=\sum_{k_1+k_2+...+k_r=k,\;k_i\geq0}\prod_{i=1}^{r}\binom{n+k_i-1}{k_i}.
 \end{equation}
 To calculate this sum, suppose that we have $rn$ white points and mark $k$ from
 them, \emph{but now every point could be marked several times.} According to well
  known formula for the number of the combination with repetitions (cf. \cite{8},
   p.10), this we can do in $\binom{rn+k-1}{k}$ ways. On the other hand, we can
   mark (with repetitions) $k_1$ from $n$ points, $k_2$ from another $n$ points,
   etc. This leads us to the equality
  \begin{equation}\label{44}
 \Sigma_3=\binom{rn+k-1}{k}.
 \end{equation}
 To calculate  the required interior sum in (\ref{42})
 \begin{equation}\label{45}
\Sigma_4=\sum_{k_1+k_2+...+k_r=k,\;k_i\geq1}\prod_{i=1}^{r}\binom{n+k_i-1}{k_i},
 \end{equation}
 we should remove zero parts in $\Sigma_3$ (\ref{44}), using "include-exclude"
  formula. We find
  $$ \Sigma_4=\binom{rn+k-1}{k}-\binom{r}{1}\binom{(r-1)n+k-1}{k}+$$
\begin{equation}\label{46}
 \binom{r}{2}\binom{(r-2)n+k-1}{k}-...+(-1)^{r-1}\binom{r}{r-1}\binom{n+k-1}{k}.
 \end{equation}
Now in an analogous way, as in Section 2, we find that the identity (\ref{42})
 is equivalent to the identity dual to (\ref{13}):
\begin{equation}\label{47}
\sum_{i=1}^k(-1)^i\binom{ni+k-1}{k}\binom{k+1}{i+1}=(-1)^k\binom{n}{k}.
 \end{equation}
The latter identity is easily proved as (\ref{13}) in Section 7.
\end{proof}
\section{An algebraic approach}
L. Tevlin \cite{10} outlined the contours of quite another proof of Theorem \ref{t3}
in frameworks of the good advanced theory of symmetric functions.
Recall (cf.\cite{6}, \cite{8}) that, for each integer $k\geq0,$\newline
$i)$ the $k$-th elementary symmetric function $e_k$ is the sum of all products
 of $k$ distinct variables $x_i,$ so that $e_0=1$ and, for $k\geq1,$
\begin{equation}\label{48}
e_k=\sum_{i_1<i_2<...<i_k}x_{i_1}x_{i_2}...x_{i_k};
 \end{equation}
$ii)$ the $k$-th complete symmetric function $h_k$ is defined as $h_0=1$ and,
for $k\geq1,$
\begin{equation}\label{49}
h_k=\sum_{i_1\leq i_2\leq...\leq i_k}x_{i_1}x_{i_2}...x_{i_k}.
 \end{equation}
 \newpage
In particular, $h_1=e_1.$ It is convenient to define
 $ e_k=h_k=0$ for $k<0;$\newline
 $iii)$ it is well known that
 \begin{equation}\label{50}
h_k= \left |\begin{matrix}e_1&e_2&
e_3& \ldots &e_{k-2}&e_{k-1}& e_k\\ 1 & e_1 & e_2&\ldots
&e_{k-3}& e_{k-2}& e_{k-1} \\ 0 & 1 & e_1& \ldots &
e_{k-4}& e_{k-3} & e_{k-2}\\
0 & 0 & 1 & \ldots & e_{k-5} &e_{k-4} & e_{k-3}\\ \ldots & \ldots & \ldots & \ldots & \ldots & \ldots\\ 0 & 0 & 0 & \ldots & 1 &
e_1 & e_2\\ 0 & 0 & 0 & \ldots & 0 & 1 & e_1 \end{matrix}\right|
\end{equation}
and
\begin{equation}\label{51}
e_k= \left |\begin{matrix}h_1&h_2&
h_3& \ldots &h_{k-2}&h_{k-1}& h_k\\ 1 & h_1 & h_2&\ldots
&h_{k-3}& h_{k-2}& h_{k-1} \\ 0 & 1 & h_1& \ldots &
h_{k-4}& h_{k-3} & h_{k-2}\\
0 & 0 & 1 & \ldots & h_{k-5} &h_{k-4} & h_{k-3}\\ \ldots & \ldots & \ldots & \ldots & \ldots & \ldots\\ 0 & 0 & 0 & \ldots & 1 &
h_1 & h_2\\ 0 & 0 & 0 & \ldots & 0 & 1 & h_1 \end{matrix}\right|_.
\end{equation}
Diagonals of these determinants has very simple cycle structure that allows to give
an explicit formulas for them.
\begin{lemma}\label{L7} The following formulas hold
\begin{equation}\label{52}
h_k=\sum_{r=1}^k(-1)^{k-r}\sum_{k_1+k_2+...+k_r=k,\;k_i\geq1}
\prod_{i=1}^{r}e_{k_i};
\end{equation}
\begin{equation}\label{53}
e_k=\sum_{r=1}^k(-1)^{k-r}\sum_{k_1+k_2+...+k_r=k,\;k_i\geq1}
\prod_{i=1}^{r}h_{k_i}.
\end{equation}
\end{lemma}
\begin{proof} Consider products of nonzero elements of cycles of the Toeplitz matrix
(\ref{50}). It is easy to see that for every cycle of length 1 this product
 is $e_1,$ for every cycle of length 2 this product is $e_2,$ ..., for every cycle
 of length $i$ this product is $e_i.$ Therefore, a diagonal with cycles of length
 $k_1,k_2,...,k_r,$ such that $k_1+k_2+...+k_r=k,$ has the product of its elements
 $\prod_{i=1}^{r} e_{k_i}$ and in the determinant this product appears with sign
 $(-1)^{k-r}.$ Hence, (\ref{52}) follows. Dually we have also (\ref{53}).
 \end{proof}
 However, by Ex.1, p.26 in \cite{6}, it follows that, if $e_k=\binom{n}{k},$ then
 $h_k=\binom{n+k-1}{k}.$ Thus, in view of Lemma \ref{L7}, we obtain new proofs of identities
 (\ref{5}) and (\ref{42}).
 \section{Other identities generated by compositions of integers}
 In \cite{6} we find also other pairs $\{e_k, h_k\}$ given by explicit formulas.
 So we obtain other interesting identities generated by compositions of integers.
 We restrict ourself by the following five pairs of identities.\newline
 \newpage
 1) Pair $e_k=\frac{a(a-k)^{k-1}}{k!},\; h_k=\frac{a(a+k)^{k-1}}{k!},\;k\geq1,$
 leads to identities:
 \begin{equation}\label{54}
\sum_{r=1}^k(-1)^{k-r}\sum_{k_1+k_2+...+k_r=k,\;k_i\geq1}
\prod_{i=1}^{r}\frac{a(a-k_i)^{k_i-1}}{k_i!}=\frac{a(a+k)^{k-1}}{k!};
\end{equation}
\begin{equation}\label{55}
\sum_{r=1}^k(-1)^{k-r}\sum_{k_1+k_2+...+k_r=k,\;k_i\geq1}
\prod_{i=1}^{r}\frac{a(a+k_i)^{k_i-1}}{k_i!}=\frac{a(a-k)^{k-1}}{k!};
\end{equation}
2) Pair $e_k=\frac{(-1)^ka^kB_k}{k!},\; h_k=\frac{a^k}{(k+1)!},\;k\geq1,$ where $B_k$
is the $k$-th Bernoulli number, leads to identities:
\begin{equation}\label{56}
\sum_{r=1}^k(-1)^{k-r}\sum_{k_1+k_2+...+k_r=k,\;k_i\geq1}
\prod_{i=1}^{r}\frac{(-1)^{k_i}a^{k_i}B_{k_i}}{k_i!}=\frac{a^k}{(k+1)!};
\end{equation}
\begin{equation}\label{57}
\sum_{r=1}^k(-1)^{k-r}\sum_{k_1+k_2+...+k_r=k,\;k_i\geq1}
\prod_{i=1}^{r}\frac{a^{k_i}}{(k_i+1)!}=\frac{(-1)^ka^kB_k}{k!};
\end{equation}
3) Pair $e_k=q^{\frac{k(k-1)}{2}}\left[\begin{matrix} n\\k\end{matrix}\right],\;
 h_k=\left[\begin{matrix} n+k-1\\k\end{matrix}\right],$ where $\left[\begin{matrix} n\\k\end{matrix}\right]$ denotes the "$q$-binomial coefficient" or Gaussian polynomial
 $$\left[\begin{matrix} n\\k\end{matrix}\right]=\frac{(1-q^n)(1-q^{n-1})...
 (1-q^{n-k+1})}{(1-q)(1-q^2)...(1-q^k)}, $$
  leads to identities:

\begin{equation}\label{58}
\sum_{r=1}^k(-1)^{k-r}\sum_{k_1+k_2+...+k_r=k,\;k_i\geq1}
\prod_{i=1}^{r}q^{\frac{k_i(k_i-1)}{2}}\left[\begin{matrix} n\\k_i
\end{matrix}\right]=\left[\begin{matrix} n+k-1\\k\end{matrix}\right];
\end{equation}
\begin{equation}\label{59}
\sum_{r=1}^k(-1)^{k-r}\sum_{k_1+k_2+...+k_r=k,\;k_i\geq1}
\prod_{i=1}^{r}\left[\begin{matrix} n+k_i-1\\k_i\end{matrix}\right]=q^{\frac{k(k-1)}{2}}\left[\begin{matrix}
 n\\k\end{matrix}\right];
\end{equation}
4) Pair $e_k=q^{\frac{k(k-1)}{2}}/\varphi_k(q),\;
 h_k=1/\varphi_k(q),$ where
 $$\varphi_k(q)=(1-q)(1-q^2)...(1-q^k), $$
  leads to identities:

\begin{equation}\label{60}
\sum_{r=1}^k(-1)^{k-r}\sum_{k_1+k_2+...+k_r=k,\;k_i\geq1}
\prod_{i=1}^{r}(q^{\frac{k_i(k_i-1)}{2}}/\varphi_{k_i}(q))=1/\varphi_k(q);
\end{equation}
\begin{equation}\label{61}
\sum_{r=1}^k(-1)^{k-r}\sum_{k_1+k_2+...+k_r=k,\;k_i\geq1}
\prod_{i=1}^{r}1/\varphi_{k_i}(q)=q^{\frac{k(k-1)}{2}}/\varphi_k(q);
\end{equation}

5) Pair $e_k=\prod_{i=1}^k\frac{a-bq^{i-1}}{1-q^i},\; h_k=\prod_{i=1}^k
\frac{aq^{i-1}-b}{1-q^i},$
 leads to identities:
 \newpage
\begin{equation}\label{62}
\sum_{r=1}^k(-1)^{k-r}\sum_{k_1+k_2+...+k_r=k,\;k_i\geq1}
\prod_{i=1}^{r}\prod_{j=1}^{k_i}\frac{a-bq^{j-1}}{1-q^j}=\prod_{i=1}^k
\frac{aq^{i-1}-b}{1-q^i};
\end{equation}
\begin{equation}\label{63}
\sum_{r=1}^k(-1)^{k-r}\sum_{k_1+k_2+...+k_r=k,\;k_i\geq1}
\prod_{i=1}^{r}\prod_{j=1}^{k_i}
\frac{aq^{j-1}-b}{1-q^j}=\prod_{i=1}^k\frac{a-bq^{i-1}}{1-q^i}.
\end{equation}
\section{Acknowledgments}
The author thanks Ira M. Gessel for private communication \cite{3}. Especially he
is grateful to Lenny Tevlin for very useful discussions which lead to writing
 the last two sections of the paper.

\end{document}